\documentclass{amsart}
\usepackage{math}
\usepackage{lineno,tikz,mathabx,mathrsfs,mathtools,amsrefs,hyperref}
\usetikzlibrary{cd,quotes,angles,decorations,arrows}

\newtheorem{mainthm}{Theorem}

\newcount\tmpcnta

\begin{document}
\title{The domino problem for hyperbolic groups}
\author{Laurent Bartholdi}
\address{Fachrichtung Mathematik, Universität des Saarlandes}
\email{laurent.bartholdi@gmail.com}
\date{May 10th, 2023}
\thanks{The author gratefully acknowledges partial support from the ERC AdG grant 101097307}
\begin{abstract}
  We prove, for every non-virtually free hyperbolic group $G$, that there is no algorithm that, given a finite collection of dominoes, determines whether the Cayley graph of $G$ may be edge-covered by these dominoes so that colours match at vertices. This answers a conjecture by Aubrun, Barbieri and Moutot and goes towards settling a long-standing conjecture of Ballier and Stein.
\end{abstract}
\maketitle

\section{Introduction}
Consider a finitely generated group $G=\langle S\rangle$. The \emph{domino problem} asks for an algorithm that, on input a finite set $A$ (the ``colours'') and a subset $\Theta\subseteq A\times S\times A$ (the ``dominoes''), determines whether every element of $G$ may be coloured by $A$ in such a manner that $(\text{colour}(g),s,\text{colour}(g s))\in \Theta$ for all $g\in G,s\in S$; namely, whether the Cayley graph of $G$ may be vertex-coloured by $A$ in such a manner that each edge carries an allowed domino.

If $G$ admits a finite-index free subgroup, then the domino problem on $G$ is decidable~\cite{muller-s:context-free-2}. It was conjectured by Ballier and Stein that this characterizes groups with decidable domino problem. Recall that a group $G$ is called \emph{virtually $\mathscr P$} if the property $\mathscr P$ holds for a finite-index subgroup of $G$:
\begin{conj}[Ballier and Stein~\cite{ballier-stein:pgdomino}]
  A finitely generated group has decidable domino problem if and only
  if it is virtually free.
\end{conj}

While this conjecture is at present (2023) still open, there has been substantial progress towards its resolution:
\begin{enumerate}
\item Among finitely presented groups, the domino problem is ``geometric'': if $G,H$ are quasi-isometric then $G$ has decidable domino problem if and only if $H$ does.
\item By a fundamental result of Berger~\cite{berger:undecidability}, the domino problem is undecidable on $\Z^2$, and therefore on every group that contains $\Z^2$ as a subgroup. This covers all non-virtually-cyclic virtually polycyclic groups~\cite{jeandel:polycyclic}.
\item Recall next that an action by permutations of a finitely generated group $H$ on a group $G$ is \emph{translation-like} if it is free and acts by piecewise right-translations~\cite{whyte:amenability}. If furthermore $H$ is finitely presented and has undecidable domino problem, then $G$ has undecidable domino problem~\cite{jeandel:translation}*{Theorem~3}; thus every group admitting a translation-like action of $\Z^2$ has undecidable domino problem.

  Every infinite, finitely generated group admits a translation-like $\Z$-action~\cite{seward:burnside}, so every group containing a subgroup of the form $H_1\times H_2$ with $H_1,H_2$ infinite and finitely generated has undecidable domino problem. This applies in particular to ``branched groups'' such as the first Grigorchuk group.
\item The $1$-skeleton of certain regular tilings of the hyperbolic plane have undecidable domino problem; from this it follows that Baumslag-Solitar groups~\cite{aubrun-kari:bs} and fundamental groups of closed surfaces~\cite{aubrun-barbieri-moutot:dominosurface} have undecidable domino problem.
\item Groups in which a graph with undecidable domino problem can be ``simulated'' have themselves undecidable domino problem. Loosely speaking, ``simulating'' means that the vertices and edges of the simulated graph are produced by an auxiliary domino problem; see~\S\ref{ss:simulations}. The ``lamplighter group'' $(\Z/2\Z)^{(\Z)}\rtimes\Z$ has undecidable domino problem~\cite{bartholdi-salo:shifts}, since it can simulate $\Z^2$.
\end{enumerate}

\emph{Word hyperbolic groups} were introduced by Gromov in~\cite{gromov:hyperbolic} as a common generalization of small cancellation groups and fundamental groups of negatively curved manifolds. They include fundamental groups of compact negatively curved Riemannian manifolds (such as surfaces of genus $\ge2$), small cancellation groups, as well as an overwhelming proportion of finitely presented groups~\cite{gromov:hyperbolic}*{\S0.2A}.
\begin{mainthm}\label{thm:main}
  Let $G$ be a word hyperbolic group. The the domino problem on $G$ is decidable if and only if $G$ is virtually free.
\end{mainthm}

See~\cite{aubrun-barbieri-moutot:dominosurface}*{\S6} for a discussion; in particular, Gromov conjectures that every non-virtually free word hyperbolic group contains an embedded copy of a surface group, and this would imply Theorem~\ref{thm:main}. In fact, it would be enough to prove that every $1$-ended (see~\S\ref{ss:reductions}) word hyperbolic group admits a translation-like action of a surface group. This is, however, not the approach that I follow, and these conjectures remain unfortunately open.

Rather, I give a general criterion (Theorem~\ref{thm:criterion}) that guarantees undecidability of the domino problem on a graph: it suffices to simulate a sequence $(\Gamma_n)_{n\in\N}$ of connected, infinite amenable graphs and contractions $\pi_n\colon\Gamma_n\to\Gamma_{n-1}$ (namely maps sending edges to edges or points) which are $q:1$ matchings for some $q\in\N$ (namely every $x\in\Gamma_{n-1}$ has precisely $q$ preimages.). This criterion is an adaptation of Kari's~\cite{kari:undecidabilitytp} proof of undecidability of the domino problem on $\Z^2$, based on the immortality problem for piecewise affine maps in the plane.

This criterion subsumes all the results mentioned above: taking all the $\Gamma_n\cong\Z$ with $\pi_n(m)=m$ recovers all groups that simulate the grid, while with $\pi_n(m)=\lfloor m/2\rfloor$ recovers all groups that simulate the hyperbolic plane.

I show, in~\S\ref{ss:proof}, how this criterion applies to a hyperbolic group. We first reduce to a $1$-ended hyperbolic group $G$. The $\Gamma_n$ are horospheres in $G$; they are connected because $G$ in $1$-ended, and have polynomial growth. The maps $\pi_n$ are, suitably parameterized, the inverse of geodesic flow. A finite set of domino tiles representing the $\Gamma_n$ and $\pi_n$ can be recovered from a construction of an aperiodic tiling by Cohen, Goodman-Strau\ss\ and Rieck~\cite{cohen-goodman-rieck:hyperbolic}.

Word hyperbolicity of a group $G$ implies strong combinatorial properties, for example that the group has a rational growth series, and therefore that its growth rate is an algebraic number; it admits a set of normal forms for its elements that is given by a regular language. In some sense, it is possible to do even better: for example, Belk, Bleak and Matucci encode $G$ by ``asynchronous transducers'' on the binary tree~\cite{belk-bleak-matucci:embed}, and in a similar spirit Cohen, Goodman-Strau\ss\ and Rieck encode the binary (and ternary) tree into the Cayley graph of $G$ by means of a subshift of finite type. There results seem independent, but are quite close in spirit, and simplify very much the argument.

\subsection{Acknowledgments}
I am very grateful to Sebasti\'an Barbieri and Chaim Goodman-Strauss for their careful reading of a preliminary version of the manuscript, and their thoughtful remarks that helped me improve the exposition.

\section{The domino problem}
We recall the definition given in the Introduction: consider a group $G$ with a finite generating set $S$. Its \emph{Cayley graph} is the graph $\Gamma(G,S)$ with vertex set $G$ and directed edge set $G\times S$; the edge $(g,s)$ starts at $g$, ends at $g s$ and has label $s$.

An instance of the domino problem on an $S$-edge-labeled graph $\Gamma$ consists of a finite set $A$ of \emph{colours} and a set $\Theta\subseteq A\times S\times A$ of \emph{tiles}. One is asked to determine whether a colour can be assigned to each vertex in such a manner that for every edge (say from $x$ to $y$, with label $s$) the triple $(\text{colour}(x),s,\text{colour}(y))$ belongs to $\Theta$.

In the case of Cayley graphs --- and this is the case that interests us here --- this can be rephrased in terms of \emph{subshifts}. Given such a $\Theta$, consider the space
\[X_\Theta\coloneqq\{a\in A^G:\forall s\in S,g\in G:(a(g),s,a(g s))\in\Theta\}.\]
It is a $G$-invariant, closed subspace of the Cantor set $A^G$, the $G$-action being given by $(h\cdot a)(g)=a(h^{-1}g)$, and since $\Theta$ is finite it is a typical example of \emph{subshift of finite type}. The algorithmic question asked, therefore, is \emph{emptiness of subshifts of finite type} given by a presentation by forbidden and allowed patterns.

Yet equivalently, the set of regular subsets of $X_\Theta$ is a Boolean algebra
with $G$-action, and the algorithmic question asked is the \emph{triviality of $G$-Boolean algebras} given by a finite presentation
\[\langle\chi_a:a\in A\mid\sum_{a\in A}\chi_a=1,\quad\chi_a\chi_b=0\;\forall a\neq b,\quad\chi_a(s\cdot\chi_b)=0\;\forall (a,s,b)\notin\Theta\rangle.\]

Decidability of the domino problem for $G$, or equivalently emptiness problem for $G$-subshifts of finite type, or triviality of finitely presented Boolean $G$-algebras, does not depend on the choice of $S$: if $S'$ is another generating set, every instance $A,\Theta$ may readily be converted into an instance $A',\Theta'$ defining an isomorphic subshift.

We may therefore allow or forbid ``patterns'' in the Cayley graph that refer to larger (but still finite) portions of the graph; and transfer ``information'' between any vertices at bounded distance in the Cayley graph, by reserving appropriate ``slots'' in the set $A$ and transferring them edge-by-edge along any chosen path. All this ``domino programming'' appears confusing at first, but is quite standard and will be used without mention throughout this text.

\subsection{Simulations}\label{ss:simulations}
We recall in simplified form the definition of simulations from~\cite{bartholdi-salo:ll}. We consider vertex- and edge-colourings specified on an edge-labeled graph $\Gamma$, and say that $\Gamma$ \emph{simulates} a graph $\Delta$ if there is a set of domino tiles for $\Gamma$ whose solutions have the following properties:
\begin{itemize}
\item they specify at each vertex $v\in\Gamma$ a collection of \emph{marked} vertices $(v,i)_{0<i\le m(v)}$ and a collection of \emph{unmarked} vertices $(v,-i)_{0\le i< u(v)}$;
\item they specify at each edge in $\Gamma$, say joining $v$ to $w$, a subgraph of the complete bipartite graph on $\{(v,i):-u(v)<i\le m(v)\}\times\{(w,j):-u(w)<j\le m(w)\}$;
\item the graph $\Delta$ is isomorphic to the graph with vertex set $\{(v,i):v\in\Gamma,0<i\le m(v)\}$ and an edge from $(v,i)$ to $(w,j)$ whenever there is a sequence of edges from $(v,i)$ to $(w,j)$ in the bipartite graphs, that traverses only unmarked vertices except at its endpoints.
\end{itemize}
There are numerous variants of the definition: the graph $\Delta$ may be oriented, unoriented, or edge-labeled (in which case the simulation must specify the edge labels). Rather than simulating \emph{one} graph, we simulate a \emph{closed space} of graphs, namely the family of all graphs whose finite subgraphs belong to a specified collection; see~\cite{bartholdi-salo:shifts}*{\S3.2} for a discussion. The main point is:
\begin{thm}[\cite{bartholdi-salo:ll}*{Theorem~3.7}]
  If $\Gamma$ simulates $\Delta$ and $\Delta$ has undecidable domino problem, then so does $\Gamma$.\qed
\end{thm}

Note that, if $\Delta$ is the Cayley graph of a finitely presented group, ``simulating $\Delta$'' is a weaker notion than ``admitting a quasi-isometric embedding of $\Delta$'' or ``admitting a translation-like action of $\Delta$'': indeed in a simulation the edges of $\Delta$ are represented by regular patterns of edges in $\Gamma$ rather than by edges of bounded length.

\section{Towers of graphs}
A \emph{tower of graphs} consists of a sequence $(\Delta_n)_{n\in\N}$ of graphs and a sequence $(\pi_n\colon\Delta_n\to\Delta_{n-1})_{n\ge1}$ of graph contractions (namely maps that send edges to edges or points). It is \emph{connected}, \emph{bounded-degree}, \emph{amenable} (see below), etc.\ if every $\Delta_n$ has this property. For $q\in\N$, it is \emph{$q$-regular} if every vertex in every $\Delta_n$ has exactly $q$ preimages under $\pi_{n+1}$. The graphs $\Delta_n$ may be oriented, unoriented, edge-labeled, etc.; in our applications, they will be unoriented, simple graphs.

For a graph $\Gamma$, we keep the notation $\Gamma$ for its vertices, and write $E(\Gamma)$ for its set of edges, each of which is an unordered pair of vertices. For a set $A\subseteq\Gamma$ of vertices, recall that its \emph{boundary} $\partial A$ is the set of edges of $\Gamma$ that have precisely one endpoint in $A$, and its \emph{neighbourhood} $\mathscr N(A)$ is the set of vertices that either belong to $A$ or are connected to $A$ by an edge. The graph $\Gamma$ is \emph{amenable} if it admits a \emph{F\o lner sequence}: a sequence $(F_k)_{k\in\N}$ of non-empty finite subsets of $\Gamma$ with
\[k\#(\partial F_k)\le\#F_k\text{ for all }k\in\N.\]
For example, if balls grow subexponentially in $\Gamma$, then a subsequence of them will constitute a F\o lner sequence.

A tower of graphs $(\Delta_n)_{n\in\N}$ may be viewed as a special kind of graph $\Delta$, in which edges come in two flavours: ``vertical'' oriented edges, and ``horizontal'' edges of the same type as those of the $\Delta_n$. The vertex set of $\Delta$ is the disjoint union of the vertex sets of the $\Delta_n$; the horizontal edges are those of the $\Delta_n$; and there is a vertical edge from every $v\in\Delta_n$ to $\pi_n(v)\in\Delta_{n-1}$. A \emph{simulation} of a tower is a simulation of this graph. 

For example, the upper half-grid $\Z\times\N$, with vertical edges pointing down, is the graph associated with a $1$-regular tower of graphs $(\Delta_n)_{n\in\N}$, in which each $\Delta_n$ is a line and each $\pi_n$ is the identity. For $q>1$, a model of the hyperbolic plane arises from the $q$-regular tower of graphs $(\Delta_n)_{n\in\N}$, in which each $\Delta_n\cong\Z$ and $\pi_n(m)=\lfloor m/q\rfloor$.

\begin{thm}\label{thm:criterion}
  Let $\Gamma$ be an edge-labeled graph, and assume that $\Gamma$ simulates a $q$-regular tower of infinite, amenable, connected, bounded-degree graphs. Then $\Gamma$ has undecidable domino problem.
\end{thm}

This covers in particular the undecidability results, due to Berger~\cite{berger:undecidability} and Kari~\cite{kari:undecidabilitytp}, of the domino problem respectively in the plane and the hyperbolic plane. The proof will occupy the remainder of this section.

\subsection{Affine maps}
The proof of Theorem~\ref{thm:criterion} relies on a reduction from a problem on affine maps in the plane. An instance consists of a disjoint family $U=U_1\sqcup\cdots\sqcup U_\ell$ of unit squares with integral corners in $\R^2$ and a family of affine maps $f_1,\dots,f_\ell\colon\R^2\righttoleftarrow$ with rational co\"efficients. Together they define a piecewise affine map $f\colon U\to\R^2$ by $f(x)=f_i(x)$ whenever $x\in U_i$.

Hooper~\cite{hooper:tmimmortality} proves that it is undecidable, given a Turing machine, whether it has an immortal configuration, namely a configuration (internal state and tape) of the machine from which it will run forever. Kari noticed that this directly translates to a statement about piecewise affine maps: the tape of a Turing machine may be represented by the ternary expansion of a pair of real numbers (left of the head on the $x$ co\"ordinate, right of the tape on the $y$ co\"ordinate), and its internal state by a number in $\{1,\dots,\ell\}$, so the configuration of a Turing machine may be represented by an element of $U$. The one-step evolution is governed by the modification of the internal state, the value read and written under the head (realized as translations), tape movement (realized by multiplication by $(\begin{smallmatrix}3&0\\0&1/3\end{smallmatrix})$ or its inverse) so all in all by a piecewise affine map $f$. We deduce:
\begin{thm}[\cite{kari:undecidabilitytp}*{\S2}]\label{thm:hooper}
  It is undecidable, given a family of squares $U\subset\R^2$ and a piecewise affine map $f\colon U\to\R^2$ with rational co\"efficients, whether there exists an ``immortal point'': an infinite orbit $(x_0,x_1,\dots)\in U^\N$ with $f(x_i)=x_{i+1}$.\qed
\end{thm}

We are ready to embark in the proof of Theorem~\ref{thm:criterion}. It consists of three steps, described in the next three subsections:
\begin{enumerate}
\item given an instance $f\colon U\dashrightarrow U$ of the immortality problem, construct an associated set $\Theta_f$ of domino tiles;
\item prove that if $f$ has an immortal point then $\Theta_f$ admits a solution;
\item prove that if $\Theta_f$ admits a solution then $f$ has an immortal point.
\end{enumerate}

\subsection{\boldmath From a piecewise affine map $f$ to a set $\Theta_f$ of domino tiles}
Let $U=U_1\sqcup\cdots\sqcup U_\ell$ be a disjoint collection of squares and let $f\colon U\to\R^2$ be a piecewise affine map with rational co\"efficients. Write furthermore $U_i=[x_i,x_i+1]\times[y_i,y_i+1]$. We are given a set of tiles $\Theta'$ that admit solutions in $\Gamma$, and each of whose solutions describes a tower of graphs $(\Delta_n)_{n\in\N}$. Write as above $\Delta$ as the disjoint union of the $\Delta_n$, to which vertical edges representing the contractions $\Delta_n\to\Delta_{n-1}$ are added. We add decorations to the tiles in $\Theta'$, to obtain the tileset $\Theta_f$, as follows. 

Firstly, every vertex $v\in\Delta$ stores an \emph{instruction} $i(v)\in I\coloneqq\{1,\dots,\ell\}$. Since the $\Delta_n$ are connected, rules in $\Theta_f$ guarantee that the instruction is constant on every $\Delta_n$.

Secondly, every vertex stores a \emph{data} vector $d(v)=(d_x(v),d_y(v))\in U\cap\Z^2$, with the constraint that $d_x(v)\in\{x_{i(v)},x_{i(v)}+1\}$ and $d_y(v)\in\{y_{i(v)},y_{i(v)}+1\}$. Evidently only finitely many values may appear in $U\cap\Z^2$.

Thirdly, every horizontal edge $\{v,w\}$ of $\Delta$ stores a \emph{carry} vector $c(v,w)=(c_x(v,w),c_y(v,w))\in C\subset\R^2$, whose range will be further narrowed later. Horizontal edges can be given two orientations, and we require $c(v,w)=-c(w,v)$.

We impose, via domino tiles, an \emph{affine constraint}: at every vertex $v\in\Delta_n$ with $\pi_{n+1}^{-1}(v)=\{v_1,\dots,v_q\}$ we have
\begin{equation}\label{eq:affineconstraint}
  f_{i(v)}(d(v))=\frac{d(v_1)+\dots+d(v_q)}q+\sum_{\{v,w\}\in E(\Delta_n)}c(v,w).
\end{equation}

Note that the rationals involved in~\eqref{eq:affineconstraint} have bounded denominator, namely the least common multiple $M$ of $q$ and the denominators of the $f_i$. We therefore constrain the carries $c(v,w)$ to be rationals of the form $\frac1M\Z^2\cap[-L,L]^2$ for some constant $L$ to be determined later. This will limit the carries to a finite set $C$.

All in all, we have specified a finite amount of data at each vertex of $\Delta$, and therefore at each vertex of $\Gamma$; and only local rules restricting these data.

\subsection{\boldmath From $f$-orbits to solutions of $\Theta_f$}
Let $(u_n)_{n\in\N}$ be an orbit of $f$, say $u_n\in U_{i(n)}$ and $f_{i(n)}(u_n)=u_{n+1}$. We use it to construct a valid colouring of $\Delta$.

The graphs $\Delta_n$ are infinite and connected, so by~\cite{cohen-goodman-rieck:hyperbolic}*{Theorem~4.2}, there exists for all $n\in\N$ a translation-like action of $\Z$ on $\Delta_n$; namely, a bijection between the vertices of $\Delta_n$ and $T_n\times\Z$ for some subset $T_n\subset\Delta_n$ of $\Z$-orbit representatives, in such a manner that there is a path of length $\le3$ between $(t,m)$ and $(t,m+1)$ for all $t\in T_n,m\in\Z$.

Recall that, for a real number $\xi\in\R$, its \emph{$\partial$Beatty sequence} is the bi-infinite sequence
\[B_\xi(m)=\lfloor(m+1)\xi\rfloor-\lfloor m\xi\rfloor.\]
It has the properties that it takes values in $\{\lfloor\xi\rfloor,\lceil\xi\rceil\}$ and its sum over any interval of size $m$ is within $1$ of $m\xi$. For $u=(u_x,u_y)\in U$, we define its $\partial$Beatty sequence co\"ordinatewise, as $B_u(m)=(B_{u_x}(m),B_{u_y}(m))$.

Returning to our orbit $(u_n)$, we convert it to a decoration of $\Delta$ as follows. The instructions on each $\Delta_n$ are set to be $i(n)$. The data at $(t,m)\in\Delta_n\cong T_n\times\Z$ is set to be $d(t,m)\coloneqq B_{u_n}(m)$; namely, we put the $\partial$Beatty sequence of $u_n$ on every copy of $\Z$ in $\Delta_n$, centered at an element of $T_n$. Note that they take values in the corners of $U_{i(n)}$ as required. It remains to show that the carries can be specified so as to satisfy~\eqref{eq:affineconstraint}; this will require choosing the constant $c$ in the definition of $C=\frac1M\Z^2\cap[-L,L]^2$.

For $v\in\Delta_n$, define its \emph{defect}
\begin{equation}\label{eq:error}
  D(v)=(D_x(v),D_y(v))\coloneqq f_{i(v)}(d(v))-\frac{d(v_1)+\dots+d(v_q)}q.
\end{equation}
Since $U$ and $f(U)$ are bounded, the defects are bounded; and have denominator at most $M$ in each co\"ordinate. If $L\in\N$ is big enough, then we will have $M D_x(v)+L\in[0,2L]\cap\Z$ for all $n,v\in\Delta_n$, and similarly $M D_y(v)+L\in[0,2L]\cap\Z$.

We now show how the $x$-co\"ordinates $c_x(u,v)$ of the carries can be selected; the same argument should be performed on the $y$-co\"ordinates. We construct a bipartite graph $\Xi$, with black and white vertices; the white vertex set is $\{(v,-i):v\in\Delta_n,1\le i\le M D_x(v)+L\}$, and the black vertex set is $\Delta_n\times\{1,\dots,L\}$. There is an edge between $(v,-i)$ and $(w,j)$ whenever $v=w$ or $\{v,w\}$ is an edge of $\Delta_n$, namely $w\in\mathscr N(\{v\})$, the neighbourhood of $v$.

Recall that a \emph{matching} is a bijection within $\Xi$ between the black and white vertices. If there is a matching $\Phi$, then one may set
\[c_x(v,w)=\frac{\#\{\text{edges $(v,-i)\leftrightarrow(w,j)$ in }\Phi\}-\#\{\text{edges $(w,-i)\leftrightarrow(v,j)$ in }\Phi\}}M\]
and note that~\eqref{eq:affineconstraint} is satisfied.

To prove the existence of a matching, we apply Hall's marriage theorem~\cite{hall:subsets}: we are to show that, for any finite monochromatic set of vertices, there is an at least as large set of vertices of the other colour that is connected to it.

Consider a collection $B$ of white vertices (the argument is essentially the same starting with black vertices). Since the edges of $\Xi$ form a collection of complete bipartite graphs, we may assume without loss of generality that $B$ is of the form $\{(v,-i):v\in A,1\le i\le M D_x(v)+L\}$ for some finite subset $A$ of $\Delta_n$; indeed if $(v,-i)\in B$, we may add to $B$ all $(v,-i')$ increasing $B$ while not changing its neighbourhood in $\Xi$.

The intersection of $A$ with every line $\{t\}\times\Z$, for $t\in T_n$, consists of a finite collection of intervals in $\Z$, each bounded by two endpoints in $\mathscr N(A)\setminus A$. Therefore, the total number of intervals is at most $\#(\mathscr N(A)\setminus A)$. On each interval $\{a,a+1,\dots,b\}$, by the definition of $\partial$Beatty sequences, we have $\|d(t,a)+\dots+d(t,b)-(b-a+1)u_n\|_\infty<1$, so
\begin{gather*}
  \bigg\|\sum_{v\in A}d(v)-(\#A)u_n\bigg\|_\infty\le\#(\mathscr N(A)\setminus A),\\
  \intertext{and because $f_{i(n)}$ is affine,}
  \bigg\|\sum_{v\in A}f_{i(v)}(d(v))-(\#A)u_{n+1}\bigg\|_\infty\le\|f_{i(n)}\|\cdot\#(\mathscr N(A)\setminus A),\\
  \intertext{where $\|f_{i(n)}\|$ denotes the norm of the linear part of $f_{i(n)}$. Similarly,}
  \bigg\|\sum_{v\in\pi_{n+1}^{-1}(A)}\frac{d(v)}q-(\#A)u_{n+1}\bigg\|_\infty\le\#(\mathscr N(A)\setminus A),\\
  \intertext{so}
  \bigg\|\sum_{v\in A}D(v)\bigg\|_\infty\le K\#(\mathscr N(A)\setminus A)
\end{gather*}
for a universal constant $K$. It follows that $B$ contains at most $L\#A+M K\#(\mathscr N(A)\setminus A)$ white vertices. These are connected to $L\#\mathscr N(A)$ black vertices, so Hall's condition is satisfied as soon as $L\ge M K$.

\subsection{\boldmath From solutions of $\Theta_f$ to $f$-orbits}
We now show conversely that if the domino problem has a solution, then there is an immortal point of $f\colon U\dashrightarrow U$. Let $\Delta$ be the tower of graphs represented by some solution of $\Theta'$, and let $(i,d,c)$ be the decoration of $\Delta$, with $i=(i_n)_{n\in\N}$, $d=(d_n)_{n\in\N}$ and $c=(c_n)_{n\in\N}$ and $i_n\in I$, $d_n\colon\Delta_n\to D$ and $c_n\colon E(\Delta_n)\to C$ the instructions, values and carries respectively.

Since the graphs $\Delta_n$ are amenable, they admit F\o lner sequences. In fact, it suffices to assume that $\Delta_0$ is amenable. Indeed, F\o lner sequences may be \emph{pulled back} by the graph contractions $\pi_n\colon\Delta_n\to\Delta_{n-1}$. More precisely,
\begin{lem}
  If $(F_k)_{k\in\N}$ is a F\o lner sequence in $\Delta_{n-1}$, then $(\pi_n^{-1}(F_{qk}))_{k\in\N}$ is a F\o lner sequence in $\Delta_n$.
\end{lem}
\begin{proof}
Since the map $\pi_n$ is $q$-to-$1$, we have $\#\pi_n^{-1}(F_k)=q\#F_k$ and $\#\pi_n^{-1}(\partial F_k)\le q^2\#\partial F_k$.
\end{proof}

Let $F=(F_k)_{k\in\N}$ be a F\o lner sequence in $\Delta_0$, and consider the real sequence $(\sum_{v\in F_k}d(v)/\#F_k)_{k\in\N}$. It is bounded, since $d$ is bounded. Let $F^{(0)}=(F^{(0)}_k)_{k\in\N}$ be a subsequence on which it converges; $F^{(0)}$ is again a F\o lner sequence. For each $n\in\N$ let $F^{(n)}=(F_k^{(n)})_{k\in\N}$ be the pull-back of $F^{(n-1)}$ to $\Delta_n$, and define $u_n\in U$ by
\[u_n=\lim_{k\to\infty}\frac1{\#F^{(n)}_k}\sum_{v\in F_k^{(n)}} d(v).\]
These limits exist, because of the tiling rules, as we will see below; but any at rate we could have re-extracted subsequences to guarantee convergence. We have
\begin{align*}
  f_{i_n}(u_n)&=\lim_{k\to\infty}\frac1{\#F^{(n)}_k}\sum_{v\in F_k^{(n)}}f(d(v))\\
              &=\lim_{k\to\infty}\frac1{\#F^{(n)}_k}\sum_{v\in F_k^{(n)}}\bigg(\sum_{w\in\pi_{n+1}^{-1}(v)}d(w)/q+\sum_{\{v,w\}\in E(\Delta_n)}c(v,w)\bigg)\\
              &=\lim_{k\to\infty}\bigg(\frac1{\#F^{(n+1)}_k}\sum_{v\in F_k^{(n+1)}}d(v)+\frac1{\#F^{(n)}_k}\sum_{\substack{\{v,w\}\in E(\Delta_n)\\\#\{v,w\}\cap F_k^{(n)}=1}}c(v,w)\bigg)
\end{align*}
because edges with $0$ or $2$ endpoints in $F_k^{(n)}$ contribute nothing to the sum, so
\[\big\|f_{i_n}(u_n)-u_{n+1}\big\|_\infty=\lim_{k\to\infty}(\max c)\frac{\#\partial F^{(n)}_k}{\#F_k^{(n)}}=\lim_{k\to\infty}\frac Lk=0.\]

We have simultaneously proven that the $u_n$ are well-defined elements of $U$ and that they form an immortal orbit of $f$.

\section{Word hyperbolic groups}
A locally finite graph $\Gamma$ is called \emph{$\delta$-hyperbolic} for a constant $\delta\in\N$ if its triangles are ``thin'': \emph{given any three vertices $x,y,z$ and any three geodesic segments $[x,y],[y,z],[x,z]$, the $\delta$-neighbourhood of $[x,y]\cup[y,z]$ contains $[x,z]$.}

A finitely generated group $G$ is called \emph{word hyperbolic} if there are $\delta\in\N$ and a generating set such that the associated Cayley graph of $G$ is $\delta$-hyperbolic; equivalently, every Cayley graph is $\delta$-hyperbolic (for a $\delta$ depending on the generating set.)

This notion was introduced by Gromov~\cite{gromov:hyperbolic}, see also~\cite{ghys-h:gromov}. It has numerous consequences, for example that $G$ is finitely presented and its word problem is solvable in linear time. It is also generic among finitely presented groups, in the sense that, with overwhelming probability, the presentation $\langle x_1,\dots,x_m\mid r_1,\dots,r_\ell\rangle$ defines a word hyperbolic group as $\ell,|r_i|\to\infty$. We shall not make use of these facts.

\subsection{Two reductions}\label{ss:reductions}
To prove the undecidability of the domino problem on non-virtually-free word hyperbolic groups, it is sufficient to prove it for an appropriate subgroup or finite extension.

Firstly, let $G=\langle S\rangle$ be a word hyperbolic group. At the cost of replacing $G$ by an overgroup containing $G$ with index $\le 2$, we may assume that all relations of $G$ have even length, or equivalently that $\Gamma(G,S)$ is bipartite. Indeed consider the subgroup $\widehat G=\langle (s,1):s\in S\rangle$ of $G\times\Z/2$. The projection to the first co\"ordinate maps onto $G$ and has kernel either $1$ (in which case $\widehat G=G$) or $\Z/2$ (in which case $\widehat G=G\times\Z/2$). Since $\widehat G$ is hyperbolic whenever $G$ is hyperbolic, and $\widehat G$ in all cases contains $G$, we consider $\widehat G$ from now on and rename it $G$.

For a second reduction, recall that the \emph{number of ends} of a finitely generated group $G$ is the supremum of the number of infinite connected components in the complement of a large ball in its Cayley graph. Hopf~\cite{hopf:ends} and independently Freudenthal~\cite{freudenthal:ends} proved that the number of ends is $0,1,2$ or $\infty$, with $0$ ends corresponding to finite groups and $2$ ends corresponding to virtually-$\Z$ groups; and Stallings~\cite{stallings:ends} proved that infinitely many ends corresponds to groups that split as free products or HNN extensions over a finite subgroup; and finally Dunwoody~\cite{dunwoody:accessibility} proved that, if $G$ is finitely presented, then this splitting stops after finitely many steps, resulting in a finite graph of groups with $({\le}1)$-ended vertex groups and finite edge groups.

Let now $G$ be a non-virtually-free word hyperbolic group. It is finitely presented, so admits such a decomposition as a graph of groups; and at least one vertex group is $1$-ended, otherwise $G$ would be virtually free. Furthermore, this vertex group is quasi-convexly embedded in $G$, so is itself word hyperbolic.

In conclusion, we may assume without loss of generality that $G$ is a $1$-ended word hyperbolic group whose Cayley graph is bipartite.

\subsection{Cone types}\label{ss:cones}
Let $G=\langle S\rangle$ be a word hyperbolic group, and assume by~\S\ref{ss:reductions} that its Cayley graph is bipartite.

Every geometric edge of $\Gamma(G,S)$ may now be oriented towards the origin: say its extremities are $g,h$; then if $\|g\|>\|h\|$ the edge has an arrow from $g$ to $h$.

The \emph{cone} of $g\in G$ is defined as
\[C(g)\coloneqq\{g h\in G:\|g h\|=\|g\|+\|h\|\};\]
equivalently, these are the vertices of $\Gamma(G,S)$ that may reach $g$ by following edges in the direction of their arrows. The \emph{cone type} consists of those $h\in G$ with $\|g h\|=\|g\|+\|h\|$, equivalently $g^{-1}C(g)$.

Let $\delta$ be a hyperbolicity constant of $\Gamma(G,S)$, and let $B=\{h\in G:\|h\|\le2\delta\}$ denote the ball of radius $2\delta$ in $G$. We decorate each vertex $g$ of $\Gamma(G,S)$ with the following finite amount of information: for every $h\in G$ with $\|h\|\le2\delta$, record the quantity $\|g h\|-\|g\|$. This information therefore defines an element $N_g\in(\Z\cap[-2\delta,2\delta])^B$. A fundamental, if easy, result (see e.g.~\cite{cannon:combstructure}*{Lemma~7.1}) is that $N_g$ determines the cone type of $g$, and the cone type of all neighbours of $g$ in its cone. In particular, there is a finite number of cone types.

The most important property, for us, is that \emph{geodesics fellow-travel or diverge}. Let us call \emph{geodesic ray} a map $r\colon\N\to G$ such that $\|r(n)\|=\|r(0)\|+n$ for all $n\in\N$. The following result is classical, see e.g.~\cite{drutu-kapovich:ggt}*{Lemma~11.75}:
\begin{lem}\label{lem:raysfellowtravel}
  Let $r,s\colon\N\to G$ be geodesic rays with $\|r(0)\|=\|s(0)\|$, and assume that $\|r(n)-s(n)\|$ is bounded. Then $\|r(n)-s(n)\|\le2\delta$ for all $n\in\N$.
\end{lem}
\begin{proof}
  Extending the rays $r,s$ by geodesic paths from $1$ to $r(0),s(0)$ respectively, we may assume $r(0)=s(0)=1$ without loss of generality. Let $K$ be a bound for all $\|r(n)-s(n)\|$. Now given $n\in\N$, consider a geodesic triangle consisting of $r([0,n+K+\delta])$, a geodesic segment $p$ connecting $r(n+K+\delta)$ to $s(n+K+\delta)$, and $s([0,n+K+\delta])$. The point $r(n)$ is at distance at most $\delta$ from $s([0,n+K+\delta])\cup p$, so at distance at most $\delta$ from $s(m)$ for some $m\in[0,n+K+\delta]$; so at distance at most $2\delta$ from $s(n)$.
\end{proof}

Every $g\neq1$ is connected by a geodesic ray from $1$; therefore every $g\in G$ has an edge with incoming arrow. Some $g\in G$ may have more than one outgoing arrow; in this case, we linearly order $S$ and define the \emph{principal outgoing edge} as the one associated with the least generator.

The set of principal outgoing edges defines a spanning tree $\Upsilon$ of $\Gamma(G,S)$, and $\|g-1\|_\Upsilon=\|g-1\|_{\Gamma(G,S)}=\|g\|$ for all $g\in G$. Furthermore, $\Upsilon$ is produced by a regular grammar~\cite{epstein-:wp}*{Theorem~3,4,5}: every vertex has one of finitely many types; and the type of a vertex determines the types of all of its descendants. (In this manner, Cannon~\cite{cannon:combstructure} proves that the growth series of $G$ is a rational function.)

\subsection{Horospheres}
Let $\beta\colon G\to\Z$ be a \emph{horofunction} (aka \emph{Busemann function}): for the edge orientation defined in~\S\ref{ss:cones}, or any pointwise limit of such edge orientations, a function satisfying $\beta(g)=\beta(h)+1$ whenever there is an edge pointing from $g$ to $h$. For example, for any $n\in\Z$ the function $\beta(g)=\|g\|-n$ is suitable for the edge orientation given above. Consider a geodesic ray $r\colon\N\to G$ and define $\beta_r(g)=\lim_{n\to\infty}\|r(n)-g\|-n$; this amounts to taking a limit of edge orientations as $1$ is ``pushed away towards $r(\infty)$''. This is also a Busemann function, and it is unbounded.

In all cases, fix once and for all a horofunction $\beta$. Its corresponding \emph{horosphere} is the subset $\beta^{-1}(0)$. We construct a simple graph $\Gamma_\beta$, called \emph{horosphere graph}, with vertex set $\{g\in\beta^{-1}(0):C(g)\text{ is infinite}\}$, and an edge joining $g$ and $h$ whenever there are geodesic rays starting at $g,h$ that remain at bounded distance (perforce $\le2\delta$ by Lemma~\ref{lem:raysfellowtravel}) from each other.

We are in effect combining two graph structures on $G$. The original edges of $\Gamma(G,S)$ are all oriented, and we call them \emph{vertical edges}, thinking of them as pointing downwards. The edges of $\Gamma_\beta$ (and of all $\Gamma_{\beta+n}$ associated with translates of $\beta$) define unoriented \emph{horizontal edges}.

The following result is not logically necessary to prove our main result (it would be enough to prove that all connected components are infinite if $\Gamma_\beta$ is infinite), but helps clarify the situation:
\begin{lem}[\cite{cohen-goodman-rieck:hyperbolic}*{Lemma~7.4}]\label{lem:connected}
  If $G$ is $1$-ended, then the graph $\Gamma_\beta$ is connected as soon as it is non-empty.
\end{lem}
\begin{proof}
  Let us suppose first that $\Gamma_\beta$ is finite, defined for the horofunction $\beta(g)=\|g\|-n$. I claim that $\Gamma_\beta$ is connected. Were it not, partition its vertices as $V_0\sqcup V_1$ in such a manner that there is no edge between $V_0$ and $V_1$. Set $C_i=\bigcup_{g\in V_i} C(g)$ for $i=0,1$. There would then be no edge of $\Gamma(G,S)$ between $C_0$ and $C_1$: otherwise, say there is an edge pointing from $h_0\in C(g_0)$ to $h_1\in C(g_1)$ with $g_i\in V_i$; then $h_1\in C(g_0)\cap C(g_1)$. We then have $G\setminus B_{G,S}(1,n-1)=C_0\sqcup C_1\sqcup\bigcup_{g\in G:\|g\|=n,C(g)\text{ finite}}C(g)$ and $G$ has at least two ends.

  If $\Gamma_\beta$ comes from an unbounded horofunction, then $\Gamma_\beta$ is a pointwise limit of finite horosphere graphs $\Gamma_{\beta_n}$, themselves connected. Consider two points $x,y\in\Gamma_\beta$, and let $n$ be large enough that they also belong to $\Gamma_{\beta_n}$. By~\cite{bestvina-mess:boundary}*{Property $(\ddagger_M)$}, there exists a path joining $x$ to $y$ in $\Gamma(G,S)$ and visiting only vertices $z$ with $\beta_n(z)\ge0$; and furthermore the length of this path is bounded by a function of $\|x-y\|_{\Gamma(G,S)}$ only. Connect every vertex on such a path down to $\Gamma_{\beta_n}$; this shows that $x,y$ are connected in $\Gamma_{\beta_n}$ by a path of length bounded by a function of $\|x-y\|_{\Gamma(G,S)}$. They therefore remain connected in the limit.
\end{proof}

To every horosphere graph $\Gamma_\beta\eqqcolon\Gamma_0$ are associated a family of horosphere graphs $\Gamma_n\coloneqq\Gamma_{\beta-n}$ associated with shifts of the horofunction $\beta$. If $\beta$ is unbounded, these graphs are all non-empty, and we restrict ourselves to that case. In that case, every $g\in G$ has a unique \emph{successor}, the vertex on the other end of the principal outgoing edge at $g$. We write $\pi\colon G\to G$ this \emph{successor} map. (If $\beta$ were bounded, e.g.\ $\beta(g)=\|g-g_0\|$, this map would be undefined at $g=g_0$.)

\begin{lem}\label{lem:picontraction}
For every $n\in\Z$ the restriction $\pi_n\colon\Gamma_n\to\Gamma_{n-1}$ of $\pi$ defines a graph contraction, namely a map from vertices to vertices which sends every edge $\{g,h\}$ either to an edge or to a single point.
\end{lem}
\begin{proof}
  If there are rays starting at $g,h$ and at bounded distance from each other, then there are also rays starting at their successors and at bounded distance from each other.
\end{proof}

Since we shall compose the maps $\pi_n$, we introduce the convenient notation $\pi_n^m$ for the map $\pi_{n-m-1}\circ\cdots\circ\pi_n\colon\Gamma_n\to\Gamma_{n-m}$. We show that the maps $\pi_n$ actually contract substantially the graph metric of $\Gamma_n$, i.e.\ that they form a \emph{hyperbolic graph system} in the terminology of~\cite{kopra-salo:spds}*{\S5}:
\begin{lem}\label{lem:contraction}
  There is a constant $L\in\N$ such that
  \[\|\pi_n^L(g)-\pi_n^L(h)\|\le\lceil\|g-h\|/2\rceil\]
  for all $g,h\in\Gamma_n$.
\end{lem}
\begin{proof}
  Consider two adjacent horizontal edges $\{g_n,h_n\}$ and $\{h_n,k_n\}$ in $\Gamma_n$, and lift them to edges $\{g_m,h_m\}$ and $\{h_m,k_m\}$ for all $m<n$. If these are all genuine edges, then $(g_m)_{m\le n}$ and $(k_m)_{m\le n}$ are arbitrarily long geodesic segments at bounded distance from each other, and this distance is at most $2\delta$ by Lemma~\ref{lem:raysfellowtravel}; so there exists geodesic segments $(g_m)_{m\in[n-j,n-i]}$ and $(k_m)_{m\in[n-j,n-i]}$ for some $0\le i<j\le n$ with same cone types at the beginning and at the end; so there exists geodesic rays starting at $g_{n-i},k_{n-i}$ and at bounded distance from each other; so $\{g_{n-i},k_{n-i}\}$ is an edge of $\Gamma_{n-i}$. Furthermore, $i$ is bounded by the number of pairs of cone types, say $L$.
\end{proof}
  
Horospheres in hyperbolic space $\mathbb H^d$ have the geometry of $(d-1)$-dimensional euclidean spheres; so if $\Gamma$ is the fundamental group of a $d$-dimensional compact hyperbolic manifold then then its horosphere graphs are finite or quasi-isometric to $\Z^{d-1}$. This is the motivation for the following key result; see analogous statements in~\cite{franks:anosov}*{Theorem~8.3} and~\cite{gromov:nilpotent}*{\S2}:
\begin{lem}\label{lem:polygrowth}
  The horosphere graphs have uniformly polynomial growth; namely, there exist constants $C,d\in\N$ such that a ball of radius $R$ in a horosphere graph contains at most $C\cdot R^d$ vertices.
\end{lem}
\begin{proof}
  We consider more closely the contractions $\pi_n$. The number of vertical edges outgoing at every vertex is at most $\#S$; so $\pi_n$ is an at most $\#S$-to-$1$ map.

  Consider now a ball of radius $2^k$ in $\Gamma_n$. By Lemma~\ref{lem:contraction}, it maps under $\pi_n^L$ into a ball of radius $2^{k-1}$ in $\Gamma_{n-L}$, and this map is at most $D$-to-$1$ for $D\coloneqq(\#S)^L$. A ball of radius $1$ in $\Gamma_n$, on the other hand, contains at most $C\coloneqq(\#S)^{2\delta}$ elements, since horizontal edges have length at most $2\delta$ in the metric of $\Gamma(G,S)$. Letting $N_k$ denote the maximal size of a ball of radius $2^k$, we obtain $N_k\le D N_{k-1}$ so $N_k\le C\cdot D^k$ and thus the size of a ball of radius $R$ is at most $C\cdot(2R)^{\log_2D}$, a uniform polynomial upper bound.
\end{proof}

\begin{lem}\label{lem:infinite}
  If $\beta$ is unbounded, then the horosphere graphs are infinite.
\end{lem}
\begin{proof}
  $\Gamma(G,S)\setminus\Gamma_\beta$ has at least two ends, both infinite if $\beta$ is unbounded.
\end{proof}

\section{Proof of Theorem~\ref{thm:main}}\label{ss:proof}
Let $G$ be a non-virtually-free word hyperbolic group. Using the reductions from~\S\ref{ss:reductions}, we may assume that $G$ is $1$-ended and has only even-length relations. Thanks to Theorem~\ref{thm:criterion}, it is enough to prove that $G$ simulates a $2$-regular tower of infinite, amenable, connected, bounded-degree graphs.

We make use of a subshift of finite type developed by Cohen, Goodman-Strauss and Rieck~\cite{cohen-goodman-rieck:hyperbolic}*{\S8}. Their objective is to construct a non-empty subshift of finite type on $G$ on which $G$ acts with trivial stabilizers; but what they really produce is a subshift of finite type $X_{\Theta'}$ all of whose configurations determine a ``populated shelling''. We extract the relevant information as follows:
\begin{thm}[\cite{cohen-goodman-rieck:hyperbolic}*{Theorem~8.12}]\label{thm:cgr}
  There exists a non-empty $G$-subshift of finite type $X$ each of whose configurations determine:
  \begin{enumerate}
  \item an unbounded horofunction $\beta\colon G\to\Z$, unique up to addition of a constant;
  \item a bounded-degree graph $\Delta$ with ``vertical'' and ``horizontal'' edges: there is a ``horofunction'' $\beta'\colon\Delta\to\Z$ such that every edge $(v,w)$ in $\Delta$ satisfies $\beta'(v)=\beta'(w)+1$ (and then the edge is directed and called vertical) or $\beta'(v)=\beta(w)$ (and then the edge is undirected and called horizontal); the horizontal subgraphs $\Delta_n\coloneqq(\beta')^{-1}(n)$ are all infinite and connected;
  \item following vertical edges, a $2:1$ graph contraction $\pi\colon\Delta\righttoleftarrow$ with $\beta'(\pi(v))=\beta'(v)-1$ for all $v$; thus and $\pi$ restricts to contractions $\Delta_n\to\Delta_{n-1}$;
  \item a quasi-isometry $\phi\colon\Delta\to G$ and a non-decreasing quasi-isometry $\psi\colon\Z\to\Z$ such that $\beta\circ\phi=\psi\circ\beta'$; so there is a constant $C\in\N$ such that the graph metrics in $\Delta$ and $\Gamma(G,S)$ are related by
    \[C^{-1}\|x-y\|-C\le|\phi(x)-\phi(y)\|\le C\|x-y\|+C\]
    and every point in $G$ is at distance at most $C$ from $\phi(\Delta)$.
  \end{enumerate}
  In particular, the horizontal subgraph $\Delta_n$ is quasi-isometric to the horosphere graph $\Gamma_{\beta-\psi(n)}$, with a quasi-isometry constant independent of $n$.
\end{thm}
\begin{proof}
  The ``populated shelling'' contains a ``population growth'' value $\Delta$, taking values in a $2$-element set; these values are constant on horospheres of $G$, and determine the sizes of the fibres of $\psi$. There is also a ``population '' $\wp\in\{1,\dots,N\}$, determining the fibres of $\phi$. There is finally a ``parent-child matching'' determining the vertical edges of $\widehat\Gamma$. Note that in their construction every vertex has $q^\Delta$ predecessors for some $q\in\{2,3\}$ rather than $2$ as in our $\Delta$; but we may always choose $q=2$, since the only place the assumption `$\log q/\log\lambda\notin\Q$' is used is in~\cite{cohen-goodman-rieck:hyperbolic}*{Corollary~9.4}, and we may insert, using a finite amount of data, $\Delta$ intermediate horospheres so as to make the vertical branching be precisely $2$.

  Furthermore, the subshift $\Sigma$ that they construct is built on top of a subshift $\Omega_S$ that encodes the cone types of $G$, and in particular ``knows'' the graphs $\Gamma_{\beta-n}$ for all $n$. They are infinite by Lemma~\ref{lem:infinite}, and connected by Lemma~\ref{lem:connected}.
\end{proof}

The tower of graphs associated with a configuration in $X$ is precisely $(\Delta_n)_{n\in\N}$, with maps $\pi_n\colon\Delta_n\to\Delta_{n-1}$ given by the parent-child matching. The only property that remains to prove is amenability of the graphs $\Delta_n$, and it follows from Lemma~\ref{lem:polygrowth}.


\end{document}